\documentclass[11]{amsart}
\usepackage{tipa}
\usepackage{amssymb}
\usepackage{amsmath}
\usepackage{amsthm}
\usepackage{amsfonts}
\usepackage{fancyheadings}

\renewenvironment{abstract}
               {\list{}{\rightmargin\leftmargin}%
                \item[\textbf{}]\relax}
               {\endlist}

\newtheorem{innercustomthm}{Theorem}
\newenvironment{customthm}[1]
  {\renewcommand\theinnercustomthm{#1}\innercustomthm}
  {\endinnercustomthm}
\usepackage{lmodern}
\usepackage{enumerate}
\usepackage{booktabs}

\date{\vspace{-5ex}}

\newlength{\boxw}
\newtheorem{theorem}{Theorem}
\newtheorem{lemma}{Lemma}
\newtheorem{corollary}{Corollary}

\begin{document} 
\title{Order Three Normalizers of 2-Groups}
\address{Department of Mathematics\\
  University of Southern California\\
  Los Angeles, California  90089-2532}
\email{sgerhard@usc.edu}
\author{Spencer Gerhardt}
\date{October 21, 2016}

\maketitle
\begin{abstract} This paper examines order three elements of finite groups which normalize no nontrivial 2-subgroup. The motivation for finding such elements arises out of a problem in modular representation theory. The question of when these elements appear in the almost simple groups was posed by G. Robinson in the context of studying 2-blocks of defect zero. For the almost simple groups, a complete classification of order three elements with this property is determined. On the basis of this result, necessary conditions are then given for the existence of such elements in a large class of finite groups. 

\end{abstract}
\section{Introduction}
\par Problem 19 on Brauer's \cite{brauer} well-known list of questions in group representation theory asks to describe the number of blocks of defect zero of a finite group $G$ in terms of group-theoretic invariants. Robinson \cite{robinson} provides a solution to this question, however in many cases the proposed invariants are difficult to determine. 
\par 
More recently, a lower bounds on 2-blocks of defect zero was obtained in terms of a separate group-theoretic property. In \cite{robinson2}, it is shown that the number of 2-blocks of defect zero of a finite group $G$ is at least as great as the number of conjugacy classes of elements of order three which normalize no nontrivial 2-subgroup of $G$. 
\par It is not difficult to construct groups containing order three elements with this property. For instance if $x\in G$ is order three, and $U$ is a maximal $x$-invariant 2-subgroup of $G$, then $xU$ will normalize no nontrivial 2-subgroup of $N_G(U)/U$, and hence $N_G(U)/U$ will have a 2-block of defect zero. However, it is less clear when such elements might arise in familiar contexts. In order to better understand this situation, Robinson asks which almost simple groups contain elements of order three normalizing no nontrivial 2-subgroup. The small list of examples is described below. 
 \begin{theorem} 
 Let $G$ be a finite almost simple group, and $x\in G$ be an element of order three. Then 
$x$ normalizes a nontrivial 2-subgroup of $G$, unless $G$ and $x$ occur in the following list
\begin{enumerate}[(i)]
\item $PSL_2(2^a)$, $a$ odd, $x$ in the unique class of order three elements;
\item $PGL_3(2^a)$, $a$ even, $x$ in an irreducible torus;
\item $PGU_3(2^a)$, $a$ odd, $x$ in an irreducible torus;
\item ${^2G}_2(q^2)$, $q^2=3^{2f+1}$, $x$ in the class $(\tilde{A_1})_3$.
\end{enumerate}
\end{theorem} 

Making use of Theorem 1, it is possible to determine necessary conditions for the existence of order three elements with the stated property in a large class of finite groups. Let $O_p(G)$ be the largest normal $p$-subgroup of $G$. Recall that a component $L$ of $G$ is a quasisimple subnormal subgroup. The following general restriction is established in the final section.
\begin{theorem} Assume $G$ is a finite group with components $L_1,...,L_n$. Furthermore, assume that $O_3(G) =1$. Let $x\in G$ be an order three element normalizing no nontrivial 2-subgroup of $G$. Then 
\begin{enumerate}[(i)]
\item if $N$ is any normal subgroup of $G$ with order prime to three, then $N$ has order prime to six;
\item $x$ normalizes every component $L_i$ of $G$;
\item  The image of $\langle L_i, x\rangle$ in $Aut(L_i)$ appears on the list of exceptions in Theorem 1, for every component $L_i$ of $G$.
\end{enumerate} 
\end{theorem}
\par
Much of the work in proving Theorem 1 goes into determining which order three elements $x\in G$ normalize a group of order two. This is equivalent to asking whether $x$ centralizes an involution in $G$. Theorem 3 classifies the order three elements of almost simple groups  which centralize no involution. Note if $G$ is a finite classical group in characteristic $p=3$, and $S$ is the corresponding simple group, then $Z(G)$ has order prime to three. Hence any order three element $x\in S$ can be lifted to a unipotent element of order three $\hat{x}\in G$.
\begin{theorem} 
Assume $G=\langle S,x\rangle$ is a finite almost simple group, with $x\in G$ an element of order three.  Then
$x$ centralizes an involution in $G$, unless $G$ and $x$ occur in Table 1.
\begin{table}[h!]
  \caption{}
    \label{tab:table1}
    \begin{tabular}{cc}
      \toprule
       G & x\\
       \midrule
      $PSL_2(2^a)$, $a\geq 1$ &  in the unique class of order three elements  \\ 
     $PSL_2(q)$, $q \equiv 5$ (mod 12)   &   in the unique class of order three elements \\
     $PSL_2(q)$, $q \equiv 7$ (mod 12)   &  in the unique class of order three elements  \\
        $PSL_3(2^a)$, $a\geq 1$ &  in a split torus for $a$  even;   \\
        & in a partially split torus for $a$ odd\\
        
       $PSU_3(2^a)$, $a\geq1$ &  in a split torus for $a$  odd; \\
        & in a partially split torus for $a$ even\\
       $PGL_3(q)$, $q\equiv 1$ (mod 3)& in an irreducible torus \\
$PGU_3(q)$, $q\equiv -1$ (mod 3)& in an irreducible torus \\
           $PSL_2(3^a)$, $a\geq 1$ & $\hat{x}$ has Jordan form $J_2$ \\
             $PSL_3(3^a)$, $a\geq 1$ & $\hat{x}$ has Jordan form $J_3$ \\
      $PSU_3(3^a)$, $a\geq1$& $\hat{x}$ has Jordan form $J_3$\\
      $PSL_4(3^a)$, $a$ odd & $\hat{x}$ has Jordan form $J_3\oplus J_1$\\
     $PSU_4(3^a)$, $a\geq 1$ & $\hat{x}$ has Jordan form $J_3\oplus J_1$\\
$\mathrm{Alt}_{5}$  & in class (1 2 3)\\
$\mathrm{Alt}_{6}$  & in class (1 2 3), or (1 2 3)(4 5 6)\\
$\mathrm{Alt}_7$ & in class (1 2 3)(4 5 6)\\
$\mathrm{Alt}_n$, $n=9,10$, & in class (1 2 3)(4 5 6)(7 8 9)\\
the Janko group $J_3$& in class 3B from the $J_3$ table in \cite{atlas}\\
 ${G}_2(q)$, $q=3^{f}$ & in class $(\tilde{A_1})_3$ from Table 22.2.6 of \cite{liebeck} \\
 ${^2G}_2(q^2)$, $q^2=3^{2f+1}$& in class $(\tilde{A_1})_3$ from Table 22.2.7 of \cite{liebeck} \\  
      \bottomrule
    \end{tabular}
\end{table}

\end{theorem}
\par Let us briefly describe the order in which the main theorems are proved. The paper begins with a short section on centralizers of order three elements in the alternating and sporadic groups. The next and most lengthy section determines all odd order centralizers of order three elements in the finite simple groups of Lie type. Results in this section are proved by working in the relevant algebraic groups, and reducing to the finite simple group case. The arguments differ substantially in characteristics $p=3$, and $p\ne 3$. Next, there is a section considering outer automorphisms of order three. Taken together, these first three sections establish Theorem 3. In the final section, Theorem 1 is proved by checking whether the list of exceptions found in Table 1 of Theorem 3 normalize some larger 2-subgroup.  Theorem 1 is then used in order to establish Theorem 2.

\section*{Acknowledgements}
The author would like to thank Geoff Robinson for suggesting the question considered in this paper, and Robert Guralnick and Gunter Malle for helpful comments and suggestions on an earlier version of this manuscript. This research was partially funded by NSF grant DMS-1265297.

\section{Alternating and Sporadic Groups} 

In this brief section we determine all order three elements with odd order centralizers in the alternating and sporadic groups. The following observation restricts our search for order three elements in $\mathrm{Alt}_n$ that centralize no involution.

 \begin{lemma} Assume $x\in \mathrm{Alt}_n$ is order three. Then $x$ centralizes an involution for $n> 12$. 
\end{lemma}
\begin{proof}
If $x$ has at least four fixed points, then $x$ will centralize a product of two disjoint 2-cycles, and hence an involution in $\mathrm{Alt}_n$. Next, an element of the form $(123)(456)(789)(10\ 11\ 12)$ is centralized by the involution $(14)(25)(36)(7 \ 10)(8\  11) (9 \ 12)$. Hence any order three element containing four or more three cycles will centralize an involution. For $n> 12$, all elements of order three in $\mathrm{Alt}_n$ contain either four fixed points or four disjoint three cycles.
\end{proof}
By an easy inspection, we are led to the following. 
\begin{lemma} Let $S=\mathrm{Alt}_n$ and assume $x\in S$ is order three. Then $C_{S}(x)$ is even order, unless $n=5,6,7,9$ or $10$.
\end{lemma}
The ATLAS \cite{atlas} lists the centralizer orders of order three elements in the sporadic groups. This yields the following conclusion.

 \begin{lemma} Assume $S$ is a sporadic group, and $x\in S$ is order three. Then $C_S(x)$ is even order, unless $S= J_3$ and $x$ is in the conjugacy class $3B$. 
\end{lemma}
\section{Groups of Lie Type}
In this section we determine the order three elements with odd order centralizers in the finite simple groups of Lie type. As noted earlier, the arguments make frequent use of the relevant algebraic groups. Unless otherwise stated, let $X$ be a simple algebraic group of simply connected  type, $X^F=G$ be the fixed points of a Steinberg endomorphism $F:X\rightarrow X$, and $S=G/Z(G)$. Then $S$ is a finite simple group, unless $G$ appears in the following list (see \cite{malle}, Theorem 24.17):

$$SL_2(2),SL_2(3), SU_3(2), Sz(2), Sp_4(2), G_2(2), {^2}G_2(3) \text{ or } {^2}F_4(2)$$

In the first four cases listed above, the group $G$ is solvable. These cases are not relevant to Theorem 3. In the latter four cases, $G/Z(G)$ is not simple but the derived subgroup $G'$ is. As $Sp_4(2)'\cong \mathrm{Alt}_6$, $G_2(2)'\cong PSU_3(3)$ and ${^2}G_2(3)'\cong PSL_2(8)$, these groups are considered elsewhere. From the ATLAS \cite{atlas}, we find ${^2}F_4(2)'$ has a unique class of order three elements, and elements in this class have even order centralizers. Hence in what follows it suffices to limit our attention to cases where $S=G/Z(G)$ is simple.
\par Using some representation, view $G$ as a matrix group over the field $k=\mathbb{F}_q$, where $q=p^n$, and $p$ is the underlying characteristic of $G$. Recall $x\in G$ is $\it{semisimple}$ if it is diagonalizable over the algebraic closure of $k$, $\it{unipotent}$ if all its eigenvalues are 1, and $\it{regular}$ if the dimension of its centralizer in $X$ is as small as possible. Note when $x\in G$ is order three and $p=3$, $(x-I)^3=x^3-I=0$, with $I$ the identity matrix. So $x$ is unipotent. Similarly when $p \ne 3$, $3\mid q^2-1$, so an order three element is diagonalizable over $\mathbb{F}_{q^2}$, and hence semisimple. It will be convenient to treat these two cases separately. 
\par 
Next, let $\pi:G\rightarrow G/Z(G)$ be the natural projection map, and $\hat{x}$ be an element in the preimage of $x\in G/Z(G)$. We will say $x\in S$ is semisimple, unipotent, or regular, if some element $\hat{x}\in G$ is. The following lemma shows that the question of whether an order three element $x\in S$ centralizes an involution can be lifted to the question of whether an odd order element $\hat{x}\in G$ commutes with a non-central 2-power element. \begin{lemma} Let $G$ be a finite group, $x\in G$ be an element of odd order, and $w$ be a $2$-power element of $G$. Then $x$ commutes with $w$ modulo $Z(G)$ if and only if $x$ commutes with $w$ in $G$.
\end{lemma}
\begin{proof} 

\par Assume $x\in G$ has odd order $r$, and $w\in G$ has order $s=2^k$, for some $k\geq1$. If $[w,x]\in Z(G)$, then $[w, x^r]=[w,x]^r=1$, and $[w^s,x]=[w,x]^s=1$. Since $r$ and $s$ are relatively prime, this implies $[w,x]=1$, and $w\in C_G(x)$. Conversely, if $w\in C_G(x)$, then $[w,x]=1\in Z(G)$.  
\end{proof} 
In particular, when $Z(G)$ has order prime to three, showing $S$ contains no order three elements with odd order centralizers is equivalent to showing all order three elements in $G$ commute with a non-central $2$-power element. This fact will be useful in treating the unipotent classes below.

\subsection{Semisimple Classes}
Assume $S$ is a finite simple group of Lie type in characteristic $p\ne 3$, and $x\in S$ is an element of order three. The following lemma shows $C_S(x)$ is even order, except in a few small rank cases. 
\begin{lemma} Let $X$ be a simply connected simple algebraic group, $X^F=G$ be the fixed points of a Steinberg endomorphism $F:X\rightarrow X$, and $S=G/Z(G)$ be a finite simple group. If $x\in S$ is semisimple but not regular semisimple, then $C_S(x)$ is even order. \end{lemma} 
\begin{proof}
Consider $\hat{x}\in X^F$. Since $X$ is simply connected and $\hat{x}$ is semisimple but not regular semisimple, $C_X(\hat{x})$ is a connected, reductive algebraic group strictly containing a torus (see \cite{humphreys}, Sections 2.2 and 2.11). Hence $C_X(\hat{x})$ contains a semisimple algebraic subgroup $Y$. The Weyl group of $Y^F$ contains elements of even order, which ensures $C_S(x)$ contains an involution.
\end{proof}
Note Lemma 5 holds more generally in cases where $G/Z(G)$ is not solvable, but $G'/Z(G')$ is simple. In this situation, if $\hat{x}\in G$ is semisimple but not regular semisimple, $C_X(\hat{x})$ contains a semisimple algebraic subgroup $Y$ and the Weyl group of $Y^F$ contains a Sylow 2-subgroup whose order does not divide $[G:G']$. Again $C_{G'/Z(G')}(x)$ has even order. 
\begin{corollary}
Let $S=G/Z(G)$ be a finite simple group, with $X$ and $X^F=G$ as above. Assume $x\in S$ is a semisimple order three element. Then $C_S(x)$ is even order, unless $S= PSL_2(q), PSL_3(q)$ or $PSU_3(q)$. 
\end{corollary}  
\begin{proof} 
To begin, assume $X= SL_n(k)$ or $Sp_n(k)$, and $V$ is the natural module for $X$. If $x\in S$ is semisimple of order three, then $\hat{x}\in G$ has at most three distinct eigenvalues on $V$. As regular semisimple elements in $X$ have no repeated eigenvalues on the natural module, $X$ contains no regular semisimple order three elements for $n\geq4$.  
\par Next let $X=Spin_n(k)$, $Y=SO_n(k)$, and $\rho:X\rightarrow Y$ be the standard covering map (so that $\rho$ is a double cover in characteristic $p\ne 2$).  Note if $y\in Y$ is not regular, then any preimage $\hat{y}\in X$ will share this property. Let $Y^F=H$ be the fixed points of a Steinberg endomorphism $F:Y\rightarrow Y$, and $R=H'/Z(H')$ be the corresponding finite simple group. A regular semisimple element in $Y$ has no repeated eigenvalues on the natural module other than $\pm 1$, and at most one eigenspace of dimension two corresponding to an eigenvalue $\pm1$. Hence $X$ and $Y$ contain no regular semisimple elements of order three for $n\geq5$. In addition, $Spin_4(k)$ is not a simple algebraic group ($Spin_4(k)/Z(Spin_4(k))\cong PSL_2(k)\times PSL_2(k)$), and $Spin_3(k)\cong Sp_2(k)\cong SL_2(k)$. 
\par Hence if $X$ is a simple classical algebraic group of simply connected type and $x\in S$ is a regular semisimple element of order three, then $X=SL_3(k)$ or $X=SL_2(k)$, and $S= PSL_2(q), PSL_3(q)$ or $PSU_3(q)$. From L\"ubeck \cite{lubeck}, there are no regular semisimple elements of order three in the simply connected exceptional algebraic groups. The conclusion follows by applying Lemma 5.
 \end{proof}
We may now classify the semisimple elements of order three with odd order centralizers.
\begin{theorem} Let $X$ be a simply connected simple algebraic group, $X^F=G$ be the fixed points of a Steinberg endomorphism $F:X\rightarrow X$, and $S=G/Z(G)$ be a finite simple group. Assume $x\in S$ is a semisimple order three element. Then $C_{S}(x)$ is even order, unless $S$ and $x$ occur in the following list.
\begin{enumerate}[(i)]
\item $PSL_3(2^a)$, $a\geq 1$, $x$ in a split torus for $a$ even; $x$ in a partially split torus for $a$ odd
\item $PSU_3(2^a)$, $a\geq 1$, $x$ in a split torus for $a$ odd; $x$ in a partially split torus for $a$ even
\item $PSL_2(2^a)$, $a\geq 1$, $x$ in a split torus for $a$ even; $x$ in an irreducible torus for $a$ odd
\item $PSL_2(q)$, $q\equiv 7$ (mod 12), $x$ in a split torus
\item $PSL_2(q)$ $q\equiv 5$ (mod 12), $x$ in an irreducible torus
\end{enumerate}
 \end{theorem} 
\begin{proof} 
Applying Lemma 5 and Corollary 1, any order three element $x\in S$ will centralize an involution unless $(a)$ $S=PSL_2(q)$, $PSL_3(q)$ or $PSU_3(q)$, and $(b)$ $x$ is regular semisimple. So assume $x$, $S$ satisfy conditions $(a)$ and $(b)$. Since $X$ is a simply connected algebraic group, $\hat{x}$ will live in some maximal torus $Q$, with $C_X(\hat{x})^F=Q$ (see \cite{humphreys}, Section 2.11). Let $T=Q^F/Z(X^F)$ be the corresponding maximal torus in $S$. Then $C_S(x)$ is odd order if and only if $T$ is. Hence to classify semisimple order three elements with odd order centralizers, it suffices to find all odd order tori in $PSL_2(q)$, $PSL_3(q)$ or $PSU_3(q)$ containing regular semisimple elements of order three. 
\par To begin, assume $S= PSL_3(q)$. Up to conjugacy, $S$ has three maximal tori; split, partially split, and irreducible. Call these tori $T_1,T_2$ and $T_3$. 
\par 
A split torus $T_1$ has order $\frac{(q-1)^2}{(3,q-1)}$, and hence is odd order if and only if $q$ is even. Furthermore, $T_1$ contains elements  of order three if and only if $q\equiv 1$ (mod 3). So assume $q\equiv 1$ (mod 3) is even, and $\alpha\in\mathbb{F}^{\times}_q$ has order three. Then some $\hat{x}\in X^F$ has eigenvalues $1,\alpha,\alpha^2$ on the natural module, and hence $T_1$ contains regular semisimple elements of order three. It follows that $S= PSL_3(2^a)$, $a$ even, contains elements of order three centralizing no involution. 
\par 
A partially split torus $T_2$ has order $\frac{q^2-1}{(3,q-1)}$, and so is odd order only if $q$ is even. First assume $q$ is even, and $q\equiv -1$ (mod 3). Let $\alpha$ be an element of order three in $\mathbb{F}_{q^2}^{\times}/\mathbb{F}_{q}^{\times}$. Some $\hat{x}\in X^F$ is diagonalizable in $SL_3(q^2)$ with eigenvalues $1,\alpha,\alpha^2$ on the natural module. In this case, $T_2$ contains regular semisimple elements of order three. On the other hand, when $q\equiv 1$ (mod 3), $T_2$ contains no regular elements of order three. So $S= PSL_3(2^a)$, $a$ odd, contains elements of order three centralizing no involution. 
\par 
An irreducible torus $T_3$ has order $\frac{q^2+q+1}{(3,q-1)}$. If $q\equiv -1$ (mod 3), then $3\nmid q^2+q+1$. If $q\equiv 1$ (mod 3), then $3\mid q^2+q+1$, but $9\nmid \frac{q^2+q+1}{(3,q-1)}$. So $T_3$ contains no elements of order three. 
\par
Next consider $S= PSU_3(q)$. The three maximal tori in $S$ have orders  $\frac{(q+1)^2}{(3,q+1)}$, $\frac{q^2-1}{(3,q+1)}$, and $\frac{q^2-q+1}{(3,q+1)}$. Repeating the arguments given above yields that $C_S(x)$ is even order, unless $S= PSU_3(2^a)$. When $a$ is odd, there are elements of order three in a split torus, and when $a$ is even, there are elements of order three in a partially split torus. 
 \par
 Finally, consider $S= PSL_2(q)$. $S$ contains two types of tori, split and irreducible, of orders $\frac{q-1}{(2,q-1)}$ and $\frac{q+1}{(2,q-1)}$. First assume $q$ is even. Then both tori have odd order and there are regular semisimple elements of order three in a split torus when $q\equiv 1$ (mod 3), and regular semisimple elements of order three in an irreducible torus when $q\equiv -1$ (mod 3). So $PSL_2(2^a)$, $a\geq 1$, yields additional classes of exceptions.  Next assume $q$ is odd. There are regular semisimple elements of order three in a split torus when $q\equiv 7$ (mod 12), and in an irreducible torus when $q\equiv 5$ (mod 12). Furthermore in both cases these tori have odd order. This yields two final classes of exceptions. \end{proof}

\subsection{Unipotent Classes}
Now let $X$ be a simply-connected simple algebraic group in characteristic $p=3$, with $X^F=G$ and $S=G/Z(G)$ as above. In this situation $Z(G)$ has order prime to three, and every order three element $x\in S$ can be lifted to a unipotent element of order three $\hat{x}\in G$. Hence by Lemma 4 to show every order three element in $S$ centralizes an involution, it suffices to show all order three elements in $G$ commute with a non-central 2-power element. 
\par 
For $X=Spin_n(k)$ it will be convenient to switch from the simply-connected algebraic group to the isogeny type $X=SO_n(k)$. Letting $X^F=G$ be as above, we have that $G/Z(G)$ is not simple, but $S=G'/Z(G')$ is. In this situation, to show every order three element $x\in S$ centralizes an involution, it suffices to show all order three elements in $G$ commute with a non-central 2-power element of spinor norm 1.
 \par 
Assume $x\in G$ is unipotent of order three. We show $C_G(x)$ contains a 2-power element with the desired properties, except in a few small rank cases. For the classical groups, the relevant criterion is whether $x$ has a repeated Jordan block. The following theorem will be used on several occasions.
\begin{theorem} Assume $X$= $GL_n(k), \ Sp_n(k)$, or $O_n(k)$ with $k$ an algebraically closed field of characteristic $p\ne 2$. Assume $x\in X$ is a unipotent element with Jordan form $\oplus_i J_i^{r_i}$. 
\begin{itemize} 
\item[$(i)$] If $X=Sp_n$, then $r_i$ is even for each odd $i$; and if $X= O_n$ then $r_i$ is even for each even $i$. 
\item[$(ii)$] Let $C_X(x)= UR$, with $U$ the unipotent radical $R_u(C_X(x))$, and $R$ the reductive part of the group. Then,
\begin{itemize}
\item[] $R= \prod GL_{r_i}$, if $X=GL_n$.
\item[] $R=\prod_{i \ odd} Sp_{r_i}\times \prod_{i\ even}O_{r_i}$, if $X=Sp_n$.
\item[] $R=\prod_{i \ odd} O_{r_i}\times \prod_{i\ even} Sp_{r_i}$, if $X=O_n$.
\end{itemize}
\end{itemize}
\end{theorem} 
\begin{proof}
See Theorem 3.1 from \cite{liebeck}.
\end{proof}
\begin{lemma} Assume $X=SL_n(k), Sp_n(k)$, or $SO_n(k)$ is a simple algebraic group, and $x\in X$ is a unipotent order three element with a repeated Jordan block. Then $C_X(x)$ contains a simple algebraic subgroup, unless $X= Sp_4(k)$, $SO_7(k)$, or $SO_8(k)$. 
\end{lemma} 
\begin{proof} 
Applying Theorem 5, if $X=SL_n(k)$ and $x\in X$ has a repeated Jordan block, then $C_X(x)$ contains a copy of $SL_2(k)$.
\par
For $X=Sp_{n}(k)$, odd sized blocks come in pairs, and any odd pair ensures the existence of $Sp_2(k)\cong SL_2(k)$ in $C_X(x)$. So we may assume $x\in X$ contains only even sized Jordan blocks. Since $x$ has order three, this further implies that $x$ contains only Jordan blocks of size two. If $x$ contains three blocks of size two, then $C_X(x)$ contains a copy of $O_3(k)\cong PGL_2(k)$. The only remaining case is $X= Sp_4(k)$, $x=J_2\oplus J_2$, where $C_X(x)/R_u(C_X(x))\cong O_2(k)$ is solvable.
\par
If $x\in SO_{n}(k)$ has a repeated Jordan block of size two, or three repeated blocks of size one or three, then $C_X(x)$ contains a copy of $SL_2(k)$. As even blocks come in pairs, we may assume that only odd blocks occur. Since $SO_2(k)$ is not simple, and $Spin_6(k)\cong SL_4(k)$, the only remaining cases to check are $X= SO_7(k)$, $x= J_3\oplus J_3\oplus J_1$, and $X= SO_8(k)$, $x=J_3\oplus J_3\oplus J_1\oplus J_1$. In these situations, $C_X(x)/R_u(C_X(x))$ is solvable. 
\end{proof} 

\begin{lemma} 
Let $X=SL_n(k), Sp_n(k)$, or $SO_n(k)$ be a simple algebraic group, with $X^F=G$ as above. Assume $x\in G$ is unipotent of order three. If $x$ contains a repeated Jordan block, then $x$ commutes with a non-central 2-power element in $G$. Furthermore, if $X=SO_n(k)$ we may assume this non-central 2-power element has spinor norm 1. 
\end{lemma} 
\begin{proof} 
Assume $x\in G$ has a repeated Jordan block, and $C_{X}(x)$ contains a simple algebraic subgroup $Y$. In characteristic $p=3$, for any such $Y\subset C_X(x)$ the order of $Y^F\subset C_G(x)$ is divisible by eight. Since $|G|/|S|\leq 4$, $C_G(x)$ must contain a non-central 2-power element (of spinor norm 1, if $X=SO_n(k)$). It then follows from Lemma 6 that the only cases requiring attention are: $X=Sp_4(k)$, $x=J_2\oplus J_2$; $X=SO_7(k)$, $x= J_3\oplus J_3\oplus J_1$; and $X=SO_8(k)$, $x=J_3\oplus J_3\oplus J_1\oplus J_1$.
\par Let $X= Sp_{4}(k)$ and $x= J_2\oplus J_2$. Table 8.1a of \cite{liebeck} lists the orders of centralizers of unipotent classes in $G=Sp_4(q)$. For $J_2\oplus J_2$ there are two classes, with centralizer orders $2q^3(q\pm1)$. Since $Z(G)$ has order two, and four divides the order of $C_G(x)$, $x$ commutes with a non-central 2-power element. 
 \par
Next assume $X=SO_7(k)$, $x= J_3\oplus J_3\oplus J_1$. Table 8.4a of \cite{liebeck} says the $SO_7(q)$ classes of $x$ have centralizer orders $2q^6(q\pm 1)$. Since the centralizer order contains a factor of four, and $S$ has index two in $SO_7(q)$, $x$ commutes with a non-central 2-power element of spinor norm 1. Finally, if $X= SO_8(k)$ and $x= J_3\oplus J_3\oplus J_1\oplus J_1$, Table 8.5a of \cite{liebeck} tells us the two $SO^{+}_8(q)$ classes of $x$ have centralizers orders $2q^8(q\pm1)^2$, and the two $SO^{-}_8(q)$ classes of $x$ have centralizer orders $2q^8(q^2-1)$. Again since eight divides the order of the centralizer and $|SO^{\pm}_8(q)|/|S|\leq 4$, $x$ commutes with a non-central 2-power element of spinor norm 1.
\end{proof} 

We are now in a position to determine the unipotent elements of order three with odd order centralizers. 

\begin{theorem} Assume $S$ is a finite simple group of Lie type in characteristic $p=3$, and $x\in S$ has order three. Then $C_{S}(x)$ is even order, unless $S$ and $x$ occur in the following list.  
\begin{enumerate}[(i)]
\item $PSL_2(3^a)$, $a\geq 1$, $\hat{x}$ has Jordan form $J_2$
\item $PSL_3(3^a)$, $a\geq 1$, $\hat{x}$ has Jordan form $J_3$
\item $PSU_3(3^a)$, $a\geq 1$, $\hat{x}$ has  Jordan form $J_3$
\item $PSU_4(3^a)$, $a\geq 1$, $\hat{x}$ has Jordan form $J_3\oplus J_1$
\item $PSL_4(3^a)$, $a$ odd, $\hat{x}$ has Jordan form $J_3\oplus J_1$
\item $G_2(q)$, $q=3^{f}$,  $x$ in class $(\tilde{A_1})^3$
\item ${^2G}_2(q^2)$, $q^2=3^{2f+1}$,  $x$ in class $(\tilde{A_1})^3$

\end{enumerate}
\end{theorem}
\begin{proof}
First assume $X$ is an exceptional algebraic group.  L\"ubeck \cite{lubeck} shows in characteristic $p=3$, all classes of order three elements have Bala-Carter parameters involving root systems of type $A_1$, $\tilde{A_1}$, $A_2$, $\tilde{A_2}$, or $G_2(a_1)$. Tables 22.2.1- 22.2.7 in \cite{liebeck} list the centralizer orders in $X^F$ of classes with such parameters. For $X\not\cong G_2(k)$, all relevant centralizers have even order. For $X^F\cong G_2(q)$ and $X^F= {^2G_2}(q^2)$, $(\tilde{A_1})^3$ is the only class of order three elements with an odd order centralizer.
\par
Next assume $X$ is a simple classical algebraic group $SL_n(k), Sp_n(k),$ or $SO_n(k)$. Let $S$ be the corresponding finite simple group, and assume $x\in S$ has order three. By lifting $x\in S$ to a unipotent element of order three $\hat{x}\in G$ and applying Lemma 7, we may assume $\hat{x}$ contains no repeated Jordan blocks, and hence that $n\leq 6$.  
\par
 First let $X=SL_n(k)$. Note if $S= PSL_n(q)$ or $PSU_n(q)$, $n>2$, and $\hat{x}$ contains a $J_2$ block, then there exists a non-central involution $diag [1,...,-1,-1,...,1]$ in $G$ commuting with $\hat{x}$ (this element occurs in $SU_n(q)$, since $3^{a}+1$ is even). Hence we may further restrict to classes in $SL_n(q)$ and $SU_n(q)$ containing no Jordan blocks of size two. This limits our search to $n\leq 4$. 
 \par
 When $n=2$ or $3$, $\hat{x}$ must have Jordan forms $J_2$ or $J_3$, respectively. In these cases, the reductive part $R$ of $C_G(\hat{x})$ is contained in $Z(G)$. Since the unipotent radical has odd order,  $C_S(x)$ will contain no involutions. Hence $PSL_2(3^a)$, $PSL_3(3^a)$, $PSU_3(3^a)$, $a\geq1$, yield classes of exceptions. 
 \par
When $n=4$, $\hat{x}$ must have Jordan form $J_3\oplus J_1$. In $SL_4(k)$, the reductive part of the centralizer of $\hat{x}$ will consist of diagonal matrices of the form $diag [\alpha,\alpha,\alpha,\beta]$, where $\alpha^3\beta=1$. Hence if $y$ is an involution modulo the center and is contained in $C_G(\hat{x})$, we must have that $y=diag [\alpha,\alpha,\alpha,\beta]$, with $\alpha^2=\beta^2$. Note that $\alpha^3\beta=1$ and $\alpha^2=\beta^2$ imply that $\alpha^8=1$. If $\alpha$ has order one, two, or four, the above conditions further imply $\alpha=\beta$, and $y$ is the identity. Hence $C_S(x)$ will have odd order, unless $\alpha$ has order eight. 
\par When $a$ is odd, $\mathbb{F}_{3^a}^{\times}$ contains no elements of order eight, so $PSL_4(3^a)$, $a$ odd, yields another class of exceptions. When $a$ is even, $\mathbb{F}_{3^a}^{\times}$ does contain an element $\alpha$ of order eight, and $diag[\alpha,\alpha,\alpha,\alpha^5]$ is a non-central involution commuting with $\hat{x}$. 
\par Elements $diag[\alpha,\alpha,\alpha,\beta]$ in $SU_4(q)$ have the additional property that the orders of $\alpha$ and $\beta$ must divide $q+1$.  Since $8\nmid q+1$, this implies that $PSU_4(3^a)$, $a\geq 1$, yields a further class of exceptions. 
\par
Next consider the symplectic groups $X= Sp_{4}(k)$ and $Sp_6(k)$. Theorem $5$ says $J_3\oplus J_1\notin Sp_4(k)$, and $J_3\oplus J_2\oplus J_1\notin Sp_6(k)$. So elements of order three will have repeated Jordan blocks, and hence centralize an involution. Finally, there are exceptional isomorphisms allowing us to handle $X= SO_n(k)$, $n\leq 6$, in terms of the other classical groups. 
\end{proof} 
Applying Lemmas 2 and 3, Theorems 4 and 6, and the Classification of Finite Simple Groups, we have determined the order three elements in the finite simple groups which centralize no involution. To complete the proof of Theorem 3, it suffices to find all order three elements in $S< G \leq \mathrm{Aut}(S)$ with odd order centralizers. 
\section{Outer Automorphisms of Order Three}
\begin{lemma} Let $G= \mathrm{Sym}_n$, with $n\geq 5$. If $x\in G$ is order three, then $C_{G}(x)$ has even order.
\end{lemma}
\begin{proof}
If $x\in G$ has two fixed points or two three cycles, it will centralize an involution in $G$. For $n\geq 5$, all order three elements have this property. 
\end{proof} 
Note for $n\geq 5$, $\mathrm{Sym}_n$ is almost simple, and for $n\ne 6$, the outer automorphism group of $\mathrm{Alt}_n$ is $\mathrm{Sym}_n$. When $n=6$, the automorphism group is slightly bigger, however this case is handled elsewhere by the exceptional isomorphism $\mathrm{Alt_6}\cong PSL_2(9)$.
\begin{lemma} Let $S< G\leq \mathrm{Aut}(S)$ be almost simple, with $S$ a sporadic group. Assume $x\in G$ is order three. Then $C_G(x)$ has even order. 
\end{lemma}
\begin{proof} 
From \cite{wilson} we find each order three element in $G$ has an even order centralizer. 
\end{proof}
\par It only remains to check order three outer automorphisms in the groups of Lie type. For this we must consider: $(i)$ diagonal automorphisms, $(ii)$ field automorphisms, $(iii)$ graph automorphisms, and $(iv)$ field-graph automorphisms. The following lemma shows the latter three cases contain no elements of order three with odd order centralizers.
\begin{lemma} 
 Let $S < G \leq \mathrm{Aut}(S)$ be an almost simple group, with $S$ a simple group of Lie type. Assume $x\in G$ is an order three field, graph, or field-graph automorphism. Then $C_{G}(x)$ has even order.
\end{lemma}
\begin{proof} 
If $S$ is a finite simple group of Lie type over the field $\mathbb{F}_{p^n}$, then the field automorphism group of $S$ is cyclic of order $n$. Hence there is an almost simple group $G=\langle S,x\rangle$ containing an order three field automorphism $x$ whenever three divides $n$. However, field automorphisms fix the prime field, and hence an even order subgroup $Y\subset S$. So $C_G(x)$  has even order. 
 \par
Order three graph and field-graph automorphisms occur only in $D_4(q)$ (see section 4.7 of \cite{gorenstein2}). In characteristic $p=3$, there are two conjugacy classes of graph automorphisms in $D_4$ (see Proposition 4.9.2 in \cite{gorenstein2}). The first has fixed point group $G_2(q)$, and the second has a parabolic subgroup of $G_2(q)$ as its fixed point group. Hence both are even order for all $q$. In characteristic $p\ne 3$, the two conjugacy classes of graph automorphisms have fixed point groups $G_2(q)$ and $A_2(q)$ (see Theorem 4.7.1 in \cite{gorenstein2}). These groups both contain involutions for all $q$. Similarly, up to diagonal automorphism, there is a single class of order three graph-field automorphisms in $D_4(q)$. The centralizer of an order three graph-field automorphism is $^3 D_4(q)$, which contains involutions for all $q$.
\end{proof} 
Finally, consider the order three diagonal automorphisms.
\begin{lemma} Let $S < G \leq \mathrm{Aut}(S)$ be an almost simple group, with $x\in G$ an order three diagonal automorphism, and $G=\langle S,x\rangle$. Then $C_G(x)$ is even order, unless $G$ and $x$ occur on the following list.
 \begin{enumerate}[(i)]
 \item $PGL_3(q)$, $q\equiv 1$ (mod 3),  $x$ in an irreducible torus
 \item $PGU_3(q)$, $q\equiv -1$ (mod 3), $x$ in an irreducible torus
 \end{enumerate} 
\end{lemma}

\begin{proof} The almost simple groups containing outer diagonal elements of order three are $E_6(q)_{ad}$, $^2E_6(q)_{ad}$, $PGL_n(q)$ and $PGU_n(q)$, with $n=3k$. From L\"ubeck \cite{lubeck}, all order three elements in $E_6(q)_{ad}$ and $^2E_6(q)_{ad}$ have even order centralizers. 
\par
So assume $G=PGL_{3k}(q)$. Since $x\in G$ is order three, $\hat{x}$ has at most three eigenvalues on the natural module. Hence for $k\geq 2$, $PGL_{3k}(q)$ contains no regular elements of order three. However when $x\in G$ is not regular, $C_G(x)$ contains a copy of $GL_2(q)$, and so in particular is even order. It follows that we may restrict our search to $\langle S,x\rangle=PGL_3(q)$, with $x$ a regular semisimple element of order three. 
\par The outer diagonal elements of order three in $PGL_3(q)$ occur when $q\equiv 1$ (mod 3), and are contained in either a split or irreducible torus. Since $x$ is regular semisimple of order three, $\hat{x}$ has three distinct eigenvalues $1,\alpha,\alpha^2$ on the natural module, with $\alpha^3=1$. In particular, $\hat{x}$ has determinant 1. It follows that if $x$ is contained in a split torus, then $x\in S$ and hence is not an outer-diagonal element.
\par 
 So we may assume $x$ occurs in an irreducible torus $T$. Let $\alpha$ be an element of order three in $\mathbb{F}_q^{\times}$. There is some $\hat{x}\in GL_3(q)$ such that $\hat{x}^3=diag[\alpha,\alpha,\alpha]$, and $\hat{x}$ has three distinct eigenvalues on the natural module (corresponding to the cube roots of $\alpha$). In this case, $x\in T$ is regular semisimple of order three, and $C_G(x)=\langle T,w\rangle$ with $w$ an element of order three. Since $T$ has order $q^2+q+1$, $C_G(x)$ has odd order. It follows that $PGL_3(q)$, $q\equiv 1$ (mod 3), contains an outer-diagonal element of order three centralizing no involution. Repeating this argument for $PGU_{3k}(q)$ yields that $PGU_3(q)$, $q\equiv -1$ (mod 3), contains outer-diagonal elements of order three with odd order centralizers. 
\end{proof} 

\section{Proofs of the Main Theorems}
Taken together, Sections $2-4$ prove Theorem 3. Now recall the statement of Theorem 1.  \begin{customthm}{1}
 Let $G$ be a finite almost simple group, and $x\in G$ be an element of order three. Then 
$x$ normalizes a nontrivial 2-subgroup of $G$, unless $G$ and $x$ occur in the following list.
\begin{enumerate}[(i)]
\item $PSL_2(2^a)$, $a$ odd, $x$ in the unique class of order three elements;
\item $PGL_3(2^a)$, $a$ even, $x$ in an irreducible torus;
\item $PGU_3(2^a)$, $a$ odd, $x$ in an irreducible torus;
\item ${^2G}_2(q^2)$, $q^2=3^{2f+1}$,  $x$ in class $(\tilde{A_1})^3$.
\end{enumerate}
\end{customthm}
To prove the theorem, it suffices to check whether the list of exceptional order three elements found in Table 1 of Theorem 3 normalize some larger 2-group.  
\subsection{Alternating and Sporadic Groups} 
Recall for $n=5,6,7,9,10$, $\mathrm{Alt}_n$ contains order three elements which centralize no involution. However by an easy inspection all such elements normalize a four group in $\mathrm{Alt}_n$. \par
Next assume $S=J_3$, and $x$ is in the conjugacy class $3B$. Note that $\mathrm{Aut}(S)\cong J_3: 2$. By \cite{wilson}, $C_{J_3:2}(x)$ has even order, and there is a single class $2A$ of outer involutions in $J_3:2$. Furthermore for $w\in 2A$, $C_{J3:2}(w)$ has order 4896.  
\par Using the list of maximal subgroups of $J_3:2$ provided in \cite{atlas}, we find that $PSL_2(17)\times 2$ has order 4896, and that it is the unique maximal subgroup of $J_3:2$ with order divisible by 4896. Hence $C_{J_3:2}(w)\cong PSL_2(17)\times 2$. 
It follows that $x$ is contained in a conjugate of $PSL_2(17)$. But $PSL_2(17)$ has a single conjugacy class of order three elements, and has $\mathrm{Alt_4}$ as a subgroup, so $x$ will normalize a four group in $J_3$. 
\subsection{Groups of Lie Type} 
To begin, let $G= PSL_2(2^a)\cong SL_2(2^a)$. A parabolic subgroup $Q$ of $G$ is a Borel subgroup, and hence has order $2^a(2^a-1)$. Furthermore, the unipotent radical $U$ of $Q$ is a Sylow-2 subgroup of $G$. Applying Borel-Tits (see Theorem 26.5 in \cite{malle}), the normalizer of any 2-group in $SL_2(2^a)$ is contained in a conjugate of $Q$. Since $3\mid 2^a(2^a-1)$ if and only if $a$ is even,  $PSL_2(2^a)$, $a$ odd, contains elements of order three normalizing no nontrivial 2-subgroup. 
\par
For $PSL_2(q)$, $q$ odd, we may assume there is a single class of order three elements. When $p\ne 3$, there is a single class, and when $p=3$ there are two classes, but they are conjugate by an outer automorphism, which preserves the property of normalizing a 2-group. Dickson's Theorem tells us $PSL_2(p)\subset PSL_2(q)$ contains $\mathrm{Alt_4}$ as a subgroup for all odd $p$, so the normalizer of a four group will contain an element from the class of order three. 
\par
For $G= PSL_3(2^a)$ and $PSU_3(2^a)$, the split and partially split tori will be contained in some proper parabolic subgroup $Q$ of $G$. The unipotent radical $U$ of the relevant parabolic is a non-trivial 2-group, with $N_G(U)=Q$. Hence all elements of order three in these tori will normalize a 2-group.
\par Next assume $G= PGL_3(q)$, $q\equiv 1$ (mod 3), with $x$ an element of order three in an irreducible torus. 
First assume $q$ is even. Then $\hat{x}\in GL_3(q)$ acts irreducibly on three dimensional space, and hence lives inside no parabolic subgroup. Since the normalizer of any 2-group in $GL_3(q)$ is contained in some parabolic, $\hat{x}$ normalizes no nontrivial 2-group in $GL_3(q)$. It follows that $x$ normalizes no nontrivial 2-subgroup in $PGL_3(2^a)$, $a$ even. The same argument shows an element of order three in an irreducible torus of $PGU_3(2^a)$, $a$ odd, will normalize no nontrivial $2$-group.
 \par 
Now assume $q$ is odd. The order of $GL_3(q)$ is $q^3(q-1)^3(q^2+q+1)(q+1)$, and the order of a split torus $T$ in $GL_3(q)$ is $(q-1)^3$. Note $3\mid q^2+q+1$, but $9\nmid q^2+q+1$. It follows that
 $N_{GL_3(q)}(T)=T\rtimes \ \mathrm{Sym_3}$ contains a Sylow-3 subgroup of $GL_3(q)$. In particular, if $K$ is a maximal elementary 2-subgroup of $T$, then $N_{GL_3(q)}(K)$ contains both $T$ and an additional order three element permuting elements in $K$. So $N_{GL_3(q)}(K)$ contains a Sylow-3 subgroup of $GL_3(q)$. By viewing the images of $K$ and $T$ in $G$, it follows that all order three elements in $PGL_3(q)$ will normalize a 2-subgroup. The same type of argument applies to elements of order three in $PGU_3(q)$, $q$ odd, $q\equiv -1$ (mod 3). 
 \par
Next let $G= PSL_3(3^a)$, and $\hat{x}=J_3$. Again, a split torus of $SL_3(3^a)$ contains an elementary abelian 2-group $K$, with $y=\left[\begin{array}{ccc}
    0 & 0 & 1         \\
    1 & 0 & 0 \\
    0 & 1 & 0
    \end{array}\right]
\in N_{SL_3(3^a)}(K)$. Since $y$ has Jordan form $J_3$,  $x$ will normalize a 2-group in $G$. Restricting to the subgroup $SL_3(3^a)\subset PSL_4(3^a)$ occurring on the $J_3$ block, the same argument can be used to show that $x$ normalizes a 2-group of $PSL_4(3^a)$, when $\hat{x}=J_3\oplus J_1$. This argument can also be repeated to show that $x$ normalizes a nontrivial 2-group in $PSU_3(3^a)$ and $PSU_4(3^a)$ in the relevant cases.  

\par Now consider $G=G_2(q)$, $q=3^f$, with $x$ in the class $(\tilde{A_1})^3$. There is a subgroup of type $A_1A_1$ in $G$ (which is the centralizer of an involution). One of the factors is generated by two short root subgroups and the other by two long root subgroups (see \cite{malle}, \S 13). The order three element $x\in (\tilde{A_1})^3$ is a diagonal product of a long and short root element. Each $A_1$ contains a copy of $\mathrm{Alt_4}$, so each long and short root element will individually normalize a four group $K$, and hence $x$ will normalize $K\times K$ in $\mathrm{Alt_4}\times\mathrm{Alt_4}$.
\par Finally consider $G={^2}G_2(q^2)$. $G$ has order $q^6(q^6+1)(q^2-1)$, $q^2=3^{2f+1}$, so its Sylow-2 subgroups have order 8. One of its Sylow-2 subgroups $P$ will be contained in ${^2}G_2(3)\cong P\Gamma L_2(8)$. Note that up to conjugacy, $P\Gamma L_2(8)$ contains a single class of involutions, and a single four group $K$. It follows that $^2G_2(q^2)$ will share this property. 
\par Let $y$ be an element from the class of involutions in $P\Gamma L_2(8)$. By Theorem 3, $N_{P\Gamma L_2(8)}(\langle y\rangle)$ has order prime to three. Furthermore, $N_{P\Gamma L_2(8)}(K)$ has order $2^3\cdot 3$, and $N_{P\Gamma L_2(8)}(P)$ has order $2^3\cdot 3\cdot 7$. So, up to conjugacy, there is a single class of order three elements in the normalizer of any 2-group. However, there are two classes of order three elements in $G$. So $^2G_2(q^2)$, $x\in (\tilde{A_1})^3$, yields a final class of exceptions.
\subsection{General Finite Groups}
We have now classified all order three elements in the almost simple groups which normalize no nontrivial 2-subgroup. This enables us to prove the following general theorem concerning finite groups. Note when $O_3(G)$ is nontrivial, it is quite messy to pin down necessary conditions for the existence of order three elements normalizing no nontrivial 2-subgroup. 

\begin{customthm}{2}
Assume that $G$ is a finite group with components $L_1,...,L_n$, and $O_3(G) =1$. Let $x\in G$ be an order three element normalizing no nontrivial 2-subgroup of $G$. Then 
\begin{enumerate}[(i)]
\item if $N$ is any normal subgroup of $G$ with order prime to three, then $N$ has order prime to six;
\item $x$ normalizes every component $L_i$ of $G$;
\item The image of $\langle L_i, x\rangle$ in $Aut(L_i)$ appears on the list of exceptions in Theorem 1, for every component $L_i$ of $G$.
\end{enumerate} 
\end{customthm}

\begin{proof}
In the following assume $O_3(G)=1$, and $x\in G$ is an order three element normalizing no nontrivial 2-subgroup. 
\\
\\
$(i)$  Assume $N$ is a normal subgroup of $G$ with order prime to three. Furthermore, assume six divides the order of $N$, and $P$ is a Sylow-2 subgroup of $N$. Then by the Frattini argument $G=NN_G(P)$. Since $N$ has order prime to 3, a conjugate of $x$ will normalize $P$. This contradicts our assumption that the class of $x$ normalizes no nontrivial 2-group. So we may assume $N$ has order prime to six.  
\\
\\
 $(ii)$ Assume $x\in G$ does not normalize some component $L_i$ of $G$, so that $L_i$, $L_i^x$ and $L_i^{x^2}$ are distinct commuting components. Let $w$ be an involution in $L_i$. Then $y=w\cdot w^x \cdot w^{x^2}$ is an involution in $G$, and $yx= xy$. 
\\
\\
$(iii)$ Pick some component $L_i$, where $L_i$ is a perfect central extension of the simple group $S_i$. Since $x$ normalizes $L_i$, $\langle L_i, x\rangle\subset  N_G(L_i)$. Furthermore, $Aut(L_i)$ can be embedded in $Aut(S_i)$. So the image of $\langle L_i,x\rangle$ in $N_G(L_i)/C_G(L_i)\subset Aut(L_i)\hookrightarrow Aut(S_i)$ must appear on the list of exceptions to Theorem 1.
\end{proof} 
Note while Theorem 2 gives necessary conditions for the existence order three elements which normalize no nontrivial 2-group, it does not claim these conditions to be sufficient. Unfortunately, no version of $(i)-(iii)$ will yield such conditions. For instance, if $G= (\mathbb{Z}_5\times \mathbb{Z}_5):\mathrm{Alt_4}$, then $G$ satisfies $(i)-(iii)$, but the single class of order three elements in $G$ normalizes a four group. In some sense this counterexample is generated by the occurrence of an order prime to six normal subgroup in $G$. However, even under the stronger assumption that no such prime to six normal subgroups appear, the conditions would still not be sufficient. 
\par For instance, let $x$ be in the single class of order three elements in $G= PSL_2(2^4)$. Let $H= G\ \wr \ \mathrm{Sym}_2$, where $\mathrm{Sym}_2$ acts by permuting the copies of $PSL_2(2^4)$. Then $y=(x,x)\in H$ is order three,  $y,H$ satisfy conditions $(i)-(iii)$ above, and $H$ contains no order prime to six normal subgroups. However $y$ centralizes the involution swapping the two copies of $PSL_2(2^4)$. Hence there is no obvious way to extend our characterization of order three elements in the almost simple groups which normalize no non-trivial 2-subgroup to a complete description of such elements in more general classes of finite groups, even under the assumption that $O_3(G)=1$.

\end{document}